\documentclass[reqno]{amsart}
\usepackage{amssymb,amsmath,epsfig,graphics,mathrsfs}

\usepackage{tabularx}
\usepackage{fancyhdr}

\usepackage{hyperref}
\hypersetup{
	colorlinks   = true,
	urlcolor     = blue,
	linkcolor    = blue,
	citecolor   = red ,
	bookmarksopen=true
}

\usepackage{dsfont}

\usepackage{cancel} 
\usepackage{esvect}
\usepackage{tikz}
\usepackage{tikz-3dplot}

\usepackage{pgfplots}
\usepackage{pgfmath}

\pgfplotsset{compat=1.18}

\usepackage{multicol}

\usepackage{amssymb,amsmath,epsfig,graphics,mathrsfs}

\usepackage{tikz}

\usepackage{graphicx}
\usepackage{caption}

\usepackage{bbm}
\usepackage[normalem]{ulem}

\usepackage[dvipsnames,table,xcdraw]{xcolor} 
\usepackage[a4paper,
left=1in,
right=1in,
top=1in,
bottom=1in,
footskip=.25in]{geometry}
\usepackage{hyperref}
\hypersetup{
	colorlinks   = true,
	urlcolor     = blue,
	linkcolor    = blue,
	citecolor   = red ,
	bookmarksopen=true
}

\def \N {\mathbb{N}}

\def \R {\mathbb{R}}

\numberwithin{equation}{section}

\newcommand{\beq}{\begin{equation}}
	\newcommand{\bea}[1]{\begin{array}{#1} }
		\newcommand{\eeq}{ \end{equation}}
	\newcommand{\ea}{ \end{array}}

\newtheorem{theorem}{Theorem}[section]
\newtheorem{lemma}[theorem]{Lemma}
\newtheorem{proposition}[theorem]{Proposition}
\newtheorem{corollary}[theorem]{Corollary}
\newtheorem{remark}[theorem]{Remark}
\newtheorem{definition}[theorem]{Definition}

\makeatletter
\def\@settitle{\begin{center}%
		\baselineskip14\p@\relax
		\bfseries
		\uppercasenonmath\@title
		\@title
		\ifx\@subtitle\@empty\else
		\\[1ex]\uppercasenonmath\@subtitle
		\footnotesize\mdseries\@subtitle
		\fi
	\end{center}%
}
\def\subtitle#1{\gdef\@subtitle{#1}}
\def\@subtitle{}
\makeatother

\usepackage[section]{placeins}

\usepackage{pgfplots}

\title[Classification of the dynamics of the 2D Keller-Segel System]{Classification of the dynamics of radial solutions to the 2D parabolic-elliptic Keller-Segel System}

\author[F.~Buseghin]{Federico Buseghin}
\address{\noindent F.~Buseghin: CY Cergy Paris University. Laboratoire AGM, 2 avenue Adolphe Chauvin 95302 Cergy-Pontoise}
\email{federico.buseghin@cyu.fr}

\author[C.~Collot]{Charles Collot}
\address{\noindent C.~Collot: CY Cergy Paris University. Laboratoire AGM, 2 avenue Adolphe Chauvin 95302 Cergy-Pontoise}
\email{charles.collot@cyu.fr}

\begin{document}

\begin{abstract}
This note gives a complete classification of the asymptotic behavior of radial solutions to the two-dimensional parabolic--elliptic Keller--Segel system on the whole space, for general initial data in the large. We review previous separate results, and unify them within a single classification framework. Depending on the mass, the flow exhibits three distinct asymptotic regimes. For a subcritical mass, solutions converge toward the unique self-similar expander of same mass. At the critical mass $8\pi$, solutions concentrate in infinite time around the stationary state with a universal logarithmic rate. The determination of this behaviour was the last missing step for achieving a complete radial classification, and we prove it in a companion paper (in fact without radial assumption). For a supercritical mass, solutions undergo type II finite-time blow-up with an explicit universal asymptotic rate and the stationary state as profile. This trichotomy holds for all radial initial data with finite second momentum. For non-radial data or infinite second momentum, it is known that other dynamics can be possible; for each of these three universal regimes we review the known results showing how they persist.

\end{abstract}
	\maketitle
\makeatletter
\def\@oddhead{\normalfont\scriptsize
\hfil CLASSIFICATION OF THE DYNAMICS OF THE 2D KELLER--SEGEL SYSTEM
\hfil\thepage}
\def\@evenhead{\normalfont\scriptsize
\thepage\hfil F.~BUSEGHIN AND C.~COLLOT\hfil}
\def\@oddfoot{}
\def\@evenfoot{}
\makeatother
	
	\section{Introduction}
    \subsection{The two dimensional Keller-Segel system on the plane}
We consider the two-dimensional Keller--Segel system
\begin{align}\label{KS}
	\begin{cases}
		\partial_t u
		=
		\Delta u-\nabla\cdot(u\nabla \Phi_u),
		\qquad (x,t)\in\mathbb R^2\times(0,\infty),\\
		-\Delta \Phi_u=u,
		\qquad x\in\mathbb R^2,\\
		u(0,x)=u_0(x), \qquad x\in\mathbb R^2.
	\end{cases}
\end{align}

The Poisson equation can be solved explicitly through the two-dimensional Newtonian potential $\Phi_u(x)=-(2\pi)^{-1}\log|\cdot|*u$, which yields
\[
\nabla \Phi_u(x)
=
-\frac1{2\pi}
\int_{\mathbb R^2}
\frac{x-y}{|x-y|^2}u(y)\,dy.
\]
Accordingly, \eqref{KS} is usually understood not as a system but as a single non-local evolution equation where $\nabla \Phi_u$ is given by the above formula.

The Keller--Segel system is widely regarded as the fundamental mathematical model for chemotactic aggregation phenomena arising in populations of microorganisms \cite{P,KS}. For general surveys on the subject we refer the reader to \cite{H,HP,BBook,SBook}. From a modeling perspective, the system describes the evolution of a population density $u$ subject to diffusion and to an attractive drift $\nabla \Phi_u$ generated by a chemical signal produced by the population itself. After a suitable rescaling, the same model can also be interpreted as describing the evolution of self-gravitating particles \cite{CS}. Since the pioneering work of Jäger and Luckhaus \cite{JL}, the mathematical analysis of \eqref{KS} has attracted considerable attention. We refer the reader to the representative works \cite{CP,NS,B,HV0,SS,BKLN,BDP,S,BCM,BCC,CHVY} concerning the interplay between diffusion and chemotactic aggregation in \eqref{KS} and related systems.

There have been numerous results on the well-posedness of Equation \eqref{KS}, see for instance \cite{BDP,BCM,BCC,W}. This raises the problem that different articles often consider different classes of initial data. Moreover, as Equation \eqref{KS} is a parabolic equation, it is natural to expect instantaneous regularization of solutions. Hence the results of some articles which are stated for a given class of initial data can be also applicable to another class of initial data, once one has shown that by parabolic regularization solutions starting in the latter instantaneously belong to the former. That is why, in order to gather together the results of various articles on large time dynamics, this note recalls and proves some well-posedness results as well as parabolic regularization results, in order to use a common space for the initial data and common spaces for estimates.

One of the most useful space for well-posedness is $L^1(\mathbb R^2)$, for which it is known that Equation \eqref{KS} is well posed \cite{N1,W}. Theorem \ref{th:lwp:L1} gives a precise statement of this fact, along with some parabolic regularization results. In particular, the solution and all its derivatives become instantaneously bounded on $\mathbb R^2$. The total mass is actually conserved along the flow,
\begin{equation}\label{id:mass-conservation}
	\int_{\mathbb R^2}u(t,x)\,dx
	=
	\int_{\mathbb R^2}u_0(x)\,dx
	=:M.
\end{equation}
The equation is invariant under translations and scaling. More precisely, if $u$ solves \eqref{KS}, then so does
\[
u_{\lambda,x_0}(t,x)
=
\frac1{\lambda^2}
u\!\left(
\frac{t}{\lambda^2},
\frac{x-x_0}{\lambda}
\right).
\]
Since $\int u_{\lambda,x_0}dx=\int u dx$, Equation \eqref{KS} is termed as mass critical. The space $L^1$ is thus the critical Lebesgue space. 

Another natural space where Equation \eqref{KS} is well-posed is $L^1(\langle x\rangle^2dx)$ which is the space associated to the norm
$$
\| u_0\|_{L^1(\langle x \rangle^2 dx)}= \int |u_0(x)|(1+|x|^2)dx,
$$
see Theorem \ref{thm:lwp-L1x2} in the present note for a precise statement. Being a subset of $L^1$, the flow satisfies all the properties of Theorem \ref{th:lwp:L1}, and possesses the following additional conserved or monotone quantities. Namely, if $u_0\in L^1(\langle x\rangle^2dx)$, the center of mass is conserved,
\begin{equation}\label{first-momentum}
	M^{-1}\int_{\mathbb R^2}xu(t,x)\,dx
	=M^{-1}
	\int_{\mathbb R^2}xu_0(x)\,dx
	=:x^*[u_0],
\end{equation}
while the second momentum satisfies the virial identity
\begin{equation}\label{second-momentum}
	\mu(t)=\int_{\mathbb R^2}|x-x^*|^2u(t,x)\,dx
	=
	\int_{\mathbb R^2}|x-x^*|^2u_0(x)\,dx
	+
	4t\left(
	1-\frac{M}{8\pi}
	\right)M.
\end{equation}
The free-energy functional
\begin{equation}\label{freeenergy}
	\mathcal F[u]
	=
	\int_{\mathbb R^2}u\log u\,dx
	-
	\frac12
	\int_{\mathbb R^2}u\Phi_u\,dx,
\end{equation}
is well defined for all $t\in (0,T(u))$ and is dissipated by the flow
\begin{equation}\label{freeenergy-2}
\mathcal F[u(t')]=\mathcal F[u(t)]-\int_{t}^{t'}u(x,s)|\nabla \ln u(x,s)-\nabla \Phi_{u(s)}(x)|^2 dx .
\end{equation}
It is intimately associated to equation \eqref{KS} because this equation is the gradient flow of the free energy functional with respect to the Wasserstein distance, see for instance \cite{O}, \cite{BlCalCar} and  \cite{EFJS}.

The second momentum identity \eqref{second-momentum} already reveals the distinguished role of the critical mass threshold $8\pi$. Indeed, the second moment measures the spatial dispersion of the density and therefore quantifies the degree of concentration of the solution. When $M<8\pi$, the second moment grows linearly in time, reflecting the dominance of diffusion and the tendency of the density to spread out. This is intimately linked with the existence of exact self-similar expanders. For all masses $0<M<8\pi$, there exists a unique profile function $\Psi_M$ such that
\begin{equation} \label{id:expander}
u(x,t)= \frac{1}{t+1} \Psi_M \!\left(\frac{x}{\sqrt{t+1}}\right) \quad \mbox{solves } \eqref{KS} \quad \mbox{and} \quad \int_{\mathbb R^2} \Psi_M dx =M,
\end{equation}
and moreover $\Psi_M$ is a radial function, see \cite{BDP,BKLN}.

In contrast, when $M>8\pi$, the second moment decreases linearly, indicating concentration on smaller and smaller spatial scales and, since as long as it is defined this quantity must be non-negative, the solution blows up in finite time. The threshold case $M=8\pi$ is therefore critical, as the second moment remains constant along the flow.

The critical mass of $8\pi$ is also important for the free energy functional \eqref{freeenergy}. A fundamental feature is its scaling behaviour and a direct computation yields
\[
\mathcal F[u_{\lambda,x_0}]
=
\mathcal F[u]
+
\left(
2M-\frac{M^2}{4\pi}
\right)\log\lambda.
\]
Hence the free energy becomes scale invariant precisely at the critical mass
$M=8\pi$. The boundedness from below at critical mass follows from the sharp logarithmic Hardy--Littlewood--Sobolev inequality established in \cite{CL}. A quantitative stability version was later discussed in \cite{CF,C}. From a variational viewpoint, the free energy is coercive for \(M<8\pi\), scale invariant at \(M=8\pi\), and unbounded from below under concentration for \(M>8\pi\). This transition reflects the change from diffusion-dominated dynamics to finite-time blow-up.

The elliptic equation corresponding to stationary solutions to \eqref{KS} also highlights the importance of the critical mass. Indeed, since it can be transformed into the Liouville equation, there exists a unique stationary state up to scaling and translation \cite{CL,BDP,BCC}, and its mass equals precisely the critical mass,
\begin{equation} \label{id:stationary-state}
U(x)=\frac{8}{(1+|x|^2)^2}, \qquad \int_{\mathbb R^2} U(x)dx=8\pi .
\end{equation}

\subsection{Classification of the dynamics for radial solutions with finite second momentum}

The aforementioned critical mass $8\pi$ plays in fact a fundamental role of \emph{threshold} for the radial dynamics: below all solutions have a global self-similar behaviour, above they blow-up with a universal anomalous self-similar behavior, and at the critical mass they concentrate in infinite time with another anomalous self-similarity. The main purpose of the present note is to gather the results of \cite{ BDP, BKLN,N1,BM, EM,M,BC} to obtain a \emph{complete classification of the asymptotic dynamics}. Analogue phase portraits are expected for critical evolution equations, which so far have only been obtained in perturbative regimes near the ground state; see \cite{MMR10} for the critical generalized KdV equation, \cite{MRS10} for the mass critical  Schr\"odinger equation and \cite{CMR} for the  critical heat equation. Here it is obtained for all data in the large, which is a remarkable achievement even in the radial case.

\begin{theorem}[Asymptotic trichotomy]
	\label{thm:classification}
	Let $u$ be a radially symmetric solution to \eqref{KS} such that $u_0\in L^1(\langle x \rangle^2 dx)$, with maximal time of existence $T$. Then the following alternatives hold.
	
	\begin{enumerate}
		
\item \textbf{Subcritical case}. If \(M<8\pi\), the solution is global \(T=\infty\) and, recalling that \(\Psi_M\) is the self-similar expander of mass \(M\) given by \eqref{id:expander}, there holds
\[
u(x,t)=\frac{1}{1+t}\Psi_M\!\left(\frac{x}{\sqrt{1+t}}\right)+\tilde u(x,t).
\]
with the following convergence rate for some universal \(\alpha>0\) ( with constant \(C=C(u_0)\))
\[
\|\tilde u(t)\|_{L^1}
+
(1+t)\|\tilde u(t)\|_{L^\infty}
\le
\frac{C}{(1+t)^\alpha},
\qquad t\ge0.
\]
\item \textbf{Critical case}. If \(M=8\pi\), the solution is global \(T=+\infty\) and can be decomposed for large times as a stationary state given by \eqref{id:stationary-state} concentrating at the origin
$$
u(x,t)=\frac{1}{\lambda^{2}(t)}U\left(\frac{x-x(t)}{\lambda(t)}\right)+\tilde u(x,t),
$$
at scale (where below $\mu$ is the second momentum \eqref{second-momentum})
$$
\lambda(t)= \sqrt{\frac{\mu}{8\pi}} \frac{1+o_{t\to \infty}(1)}{\sqrt{\ln t}},
$$
with the following convergence rate for any $\epsilon>0$ (with constant $C=C(u_0,\epsilon)$),
$$
\| \tilde u(t)\|_{L^1}+\lambda^2(t)\| \tilde u(t)\|_{L^\infty} \leq \frac{C}{(1+t)^{1-\epsilon}}.
$$

		\item \textbf{Supercritical case}. If \(M>8\pi\), then the solution blows up in finite time \(T<\infty\) and can be decomposed as a stationary state given by \eqref{id:stationary-state} concentrating at the origin
		\[
		u(x,t)
		=
		\frac1{\lambda^2(t)}
		U\!\left(\frac{x}{\lambda(t)}\right)
		+ \tilde u\!\left(x,t\right),
		\]
		at scale (where below \(\gamma\) is the Euler--Mascheroni constant)
		\[
		\lambda(t)
		=
		2e^{-1-\gamma/2}
		\sqrt{T-t}\,
		\exp\!\left(
		-\sqrt{\frac{|\log(T-t)|}{2}}
		\right)(1+o_{t\to T}(1)),
		\]
		and where
        $$
        \sup_{[0,T)} \| \tilde u(t)\|_{L^1(|x|<R)}\to 0 \quad \mbox{as }R\to 0 \qquad \mbox{and} \qquad  \lambda^2(t)\| \tilde u(t)\|_{L^\infty} \to 0 \quad \mbox{as }t\to T.
        $$
	\end{enumerate}
\end{theorem}
\begin{remark}[Partial Mass Formulation]\label{Partial Mass Formulation}
A key feature of the radial Keller-Segel equation is that in this setting a change of unknown, called the partial mass, transforms this nonlocal equation into a local one. Namely, for a radial solution \(u=u(r,t)\), one introduces
\[
M(\xi,t):=\frac{1}{2\pi}\int_{|x|<\xi}u(x,t)\,dx,
\qquad
W(\xi,t):=\frac{M(\xi,t)}{\xi^2}.
\]
Then \(W\) satisfies the local scalar parabolic equation
\[
W_t
=
W_{\xi\xi}
+
\frac{3}{\xi}W_\xi
+
W\bigl(\xi W_\xi+2W\bigr).
\]
Thus, in the radial case, the original nonlocal Keller--Segel system is reduced
to a local one-dimensional parabolic equation for the partial mass density. This
reduction is crucial for the above classification in the supercritical regime due to \cite{M}, but such a partial mass reduction is intrinsically radial and has no direct analogue
in the non-radial setting.

Equivalently, after the backward self-similar change of variables
\[
r=\frac{\xi}{\sqrt{T-t}},
\qquad
\tau=-\log(T-t),
\qquad
w(r,\tau)=(T-t)W(\xi,t),
\]
one obtains
\[
w_\tau
=
w_{rr}
+
\left(\frac{3}{r}-\frac r2\right)w_r
-
w
+
rw_rw
+
2w^2.
\]

\end{remark}

\subsection{On the universality of each of the three asymptotic behaviours beyond radial and localized data}

Theorem \ref{thm:classification} states that there are three universal behaviors for subcritical, critical and supercritical masses respectively, in the class of radial solutions with finite second momentum. These assumptions are sharp, in the sense that by weakening them the trichotomy does not hold true and new dynamics appear. We shall now comment precisely on the known classes of initial data for which each behaviour is universal.

\subsubsection{Self-similar expansion for subcritical masses}

For subcritical masses $M<8\pi$, the convergence to a self-similar expander of same mass was first proved in \cite{BDP} under the assumption of finite second momentum and $\int u \ln udx<\infty$ integrability. Then the uniqueness and radial symmetry of the self-similar expander was proved in \cite{BKLN}. Eventually, this behaviour was proved to hold for all initial data solely in $L^1$ in \cite{N1}. The robustness of this self-similar behavior of the first kind can be interpreted as the fact that self-similar expanders are the only coherent structures that can emanate from a subcritical mass, in particular no stationary state can appear. 
Considerable attention has been devoted to the well-posedness theory of the Keller--Segel system for initial data more singular than \(L^1(\mathbb R^2)\), including measure-valued data; see for instance \cite{BM}, \cite{LR}, and \cite{EFJS}. As a consequence of these developments, the asymptotic result stated above extends naturally to measure-valued initial data and yields the following theorem. 

\begin{theorem}\label{th:subcritical-measures}
Let \(u_0\in \mathcal M_+(\mathbb R^2)\) be a finite nonnegative Radon measure and assume that its total mass
\[
    M:=u_0(\mathbb R^2)<8\pi .
\]
Let \(\Psi_M\) denote the unique self-similar expander of mass \(M\) given by \eqref{id:expander}. Then the solution is global, \(T=+\infty\), and can be decomposed for large times as
\[
    u(x,t)=\frac{1}{t+1}\Psi_M\!\left(\frac{x}{\sqrt{t+1}}\right)+\widetilde u(x,t),
\]
where, as \(t\to+\infty\),
\[
    \|\widetilde u(t)\|_{L^1(\mathbb R^2)}
    +t\|\widetilde u(t)\|_{L^\infty(\mathbb R^2)}
    \to 0 .
\]
\end{theorem}

\begin{remark}[Convergence rates] For initial data in $L^1$ no convergence rate can be given, as mass can be arbitrarily distributed near spatial infinity. Under stronger assumptions on the initial datum, such as finite entropy and sufficiently high moments, sharp decay rates were obtained; see \cite{EM}.

\end{remark}

\subsubsection{Infinite time anomalous self-similarity at critical mass}

For solutions with critical mass $M=8\pi$, global existence for data in \(L^1\) was established in \cite{W}. The work \cite{BCM} previously proved global existence of free energy solutions in $L^1(\langle x\rangle^2dx)\cap \{\int u\ln u dx<\infty\}$, and that they concentrate an $8\pi$ Dirac mass at their center of mass along any time sequence converging to infinity. After the formal computations in \cite{CS}, an example of a radial and stable solution with Gaussian localization, that concentrates a stationary state in infinite time at a logarithmic scale, was given in \cite{GM}. Later, this was extended to the non-radial case with less localization assumptions in \cite{DdPDMW}. We recently proved in \cite{BC} that, without symmetry assumptions and under the even weaker and sharp sole assumption of finite second momentum, this is the universal behaviour of all critical solutions.

\begin{theorem}[\cite{BC}] \label{th:critical}

Let \(u_0\in L^1(\langle x \rangle^2 dx)\) with mass $M:=\|u_0\|_{L^1(\mathbb R^2)}=8\pi$, and center of mass $x^*$ and second momentum $\mu$ given by \eqref{first-momentum} and \eqref{second-momentum} respectively. Let $U$ denote the stationary state given by \eqref{id:stationary-state}. Then the solution is global,  $T=+\infty$, and the first and the second momenta of the solution are constant over time. Moreover, the solution can be decomposed for large times as a stationary state concentrating at its center of mass
$$
u(x,t)=\frac{1}{\lambda(t)}U\left(\frac{x-x(t)}{\lambda(t)}\right)+\tilde u(x,t)
$$
at a scale
$$
\lambda(t)= \sqrt{\frac{\mu}{8\pi}} \frac{1+o_{t\to \infty}(1)}{\sqrt{\ln t}},
$$
and with the convergence rates
$$
\| \tilde u(t)\|_{L^1(\mathbb R^2)}+\lambda^2(t)\| \tilde u(t)\|_{L^\infty(\mathbb R^2)}\leq C t^{-1+\epsilon} \quad \mbox{and} \quad |x(t)-x^*|\leq C t^{-1/2+\epsilon}
$$
as $t \to \infty$, for any $\epsilon>0$, where the constant $C>0$ depends on $\epsilon$ and $u_0$.

\end{theorem}

\begin{remark}[Convergence rates]
An algebraic decay rate of order $t^{-1}$ in $L^1$ would be the best possible decay rate in the setting of Theorem \ref{th:critical}. We expect that the approach developed in \cite{BC} could investigate the possibility of this decay rate, but this would likely require a refinement of the ansatz.
\end{remark}

Thus, the critical mass $8\pi$ is a threshold at which the self-similar spreading of the first kind of Theorem \ref{th:subcritical-measures} transitions to infinite time concentration with anomalous self-similarity. With less localization assumptions, other behaviours are known which still display an anomalous self-similarity, but there cannot be anymore universality in the scaling law.

Indeed, the large-time dynamics of the Keller--Segel equation at critical mass strongly depend on the spatial behaviour at infinity imposed on the initial datum. For solutions which are perturbative of the tail of the stationary state \eqref{id:stationary-state} at infinity, the works \cite{BCC,CF,C} (see also \cite{LGNY}) proved convergence toward a uniquely determined stationary state for critical-mass solutions with finite free energy and finite relative entropy, together with the algebraic convergence rate $t^{-1/16}$ obtained through stability estimates for the log-HLS inequality.

Without the finite relative entropy assumption, the asymptotic behavior can be drastically different. The work \cite{NaiS} constructed oscillatory radial solutions, while \cite{LGNY1} showed that bounded radial solutions either converge to a single stationary state or oscillate through an interval of stationary states, and also constructed nonradial solutions with asymptotic dynamics involving a continuum of stationary states.

As formally observed in \cite{CS}, the conservation of the second moment should be regarded as the mechanism selecting the asymptotic scale in the finite second-moment regime.

\subsubsection{Finite time type II blow-up by single or multiple collapses for supercritical mass}

\medskip

For initial data in $L^1$ with mass \(M>8\pi\), finite-time blow-up always occurs \cite{W}. The first construction of a precise blow-up example goes back to \cite{HV0,HV1,HV2,V}. This solution was then shown to be stable in the radial class in \cite{RS}, and then later \cite{CGMN0} obtained its sharp blow-up rate and non-radial stability. The work \cite{M} then proved, using comparison techniques with this precise blow-up example, that all radial solutions with data in $L^1\cap L^\infty$ produce the same blow-up behaviour. In this note we show a standard parabolic regularization result that enables to remove the $L^\infty$ assumption.

\begin{theorem} \label{th:supercritical}
	Let \(u_0\in L^1(\mathbb R^2)\) with mass $M:=\|u_0\|_{L^1(\mathbb R^2)}>8\pi$. Let $U$ denote the stationary state given by \eqref{id:stationary-state}. Then the solution blows up in finite time \(T<+\infty\). Moreover, there exist a scale \(\lambda(t)>0\) and a remainder \(\widetilde u\) such that it can be decomposed near the blow-up time as a stationary state collapsing at the origin, 
		\[
		u(x,t)
		=
		\frac1{\lambda^2(t)} U\!\left(\frac{x}{\lambda(t)}\right)+\tilde u(t,x),
		\]
        with
        \begin{equation} \label{id:scale-supercritical-result}
		\lambda(t)
		=
		2e^{-1-\gamma/2}
		\sqrt{T-t}\,
		\exp\!\left(
		-\sqrt{\frac{|\log(T-t)|}{2}}
		\right)(1+o(1)),
		\qquad t\to T,
		\end{equation}
		where \(\gamma\) denotes the Euler--Mascheroni constant, and where
		\begin{equation} \label{bd:reminder-supercritical-result}
        \sup_{[0,T)} \| \tilde u(t)\|_{L^1(|x|<R)}\to 0 \quad \mbox{as }R\to 0 \qquad \mbox{and} \qquad  \lambda^2(t)\| \tilde u(t)\|_{L^\infty} \to 0 \quad \mbox{as }t\to T.
        \end{equation}
		
\end{theorem}

The radiality assumption is essential in the mass supercritical regime. Indeed, nonradial solutions may exhibit much richer blow-up dynamics than the single-bubble concentration described in Theorem~\ref{th:supercritical}. For instance, \cite{BDdPM} constructed finite-time blow-up with multiple concentration points. Even more remarkably, \cite{CGMN2} exhibited a collision-induced blow-up mechanism producing a $16\pi$-mass singularity together with a new asymptotic rate, providing a rigorous proof of a formal argument of \cite{SSV}. These examples show that the radial asymptotic picture does not persist in the nonradial setting where the blow-up dynamics may depend on the geometry of the initial datum rather than only on conserved quantities of the flow.

\subsection{Organization of the article}
The paper is organized as follows. In Section~\ref{lwp-parReg:section} we establish local well-posedness in $L^1(\mathbb R^2)$. In Section~\ref{sec:lwp-L1x2} we study the flow in the weighted space $L^1(\langle x\rangle^2dx)$, proving the propagation of the second moment and establishing the well-definedness of the free energy functional. Finally, Section~\ref{sec:classification} is devoted to the asymptotic analysis. We first establish Theorems~\ref{th:subcritical-measures} and~\ref{th:supercritical}, and then conclude with the proof of our main result, Theorem~\ref{thm:classification}.

\subsection{Acknowledgements}

Funded by the European Union. Views and opinions expressed are however those of the author(s) only and do not necessarily reflect those of the European Union or the European Research Council Executive Agency. Neither the European Union nor the granting authority can be held responsible for them.

This work is supported by ERC grant (FloWAS, No. 101117820, 
DOI 10.3030/101117820). C. Collot was supported by the CY Initiative of Excellence Grant "Investissements d'Avenir" ANR-16-IDEX-0008 via Labex MME-DII, the grant "Chaire Professeur Junior" ANR-22-CPJ2-0018- 01 of the French National Research Agency, and the grant BOURGEONS ANR-23-CE40-0014-01 of the French National Research Agency.

\section{Notation}

The heat kernel is
$$
G(x,t)= \frac{1}{4\pi t} \exp \left( \frac{|x|^2}{4t}\right)
$$
For two nonnegative quantities $A$ and $B$, we use the notation $A\lesssim B$ to indicate that $A\leq CB$ for an absolute constant $C>0$.

\section{Local well-posedness in $L^1$ and parabolic regularization}\label{lwp-parReg:section}

In this section we establish the local well-posedness theory for the Keller--Segel system \eqref{KS} in the critical space \(L^1(\mathbb R^2)\), together with the corresponding parabolic regularization properties. Throughout this paper we consider the notion of mild solution introduced in the following definition.
	\begin{definition}\label{def:solution}
		
		Given $u_0\in L^1$, a mild solution on $[0,T)$ is a function $u$ such that
		\[
		u\in C([0,T),L^1(\mathbb R^2))
		\]
		and
		\[
		t^{1/4}u(t)\in L^\infty_{\mathrm{loc}}([0,T),L^{4/3}(\mathbb R^2)),
		\qquad
		\lim_{t\downarrow 0} t^{1/4}\|u(t)\|_{L^{4/3}}=0,
		\]
		and satisfying
		\begin{equation}\label{id:mild-solution}
			u(t)=G(t)*u_0+\int_0^t \nabla G(t-s)*\cdot\bigl(u(s)\nabla \Phi_{u(s)}\bigr)\,ds
		\end{equation}
		for every $t\in[0,T)$.
        
        Given $u_0\in L^p$ for $4/3\leq p<2$, we say that a mild solution on $[0,T)$ is a function $u\in C([0,T),L^p)$ that satisfies \eqref{id:mild-solution} for every $t\in[0,T)$.
		
	\end{definition}
	
	The notion of mild solution is only needed near the initial time. Indeed, the next result shows that every mild solution instantaneously regularizes and becomes a classical solution. This result was already introduced in \cite{BC}, but we believe that it could be useful to give a complete proof that we postpone to Section~\ref{subsec:lwp-L1}. Recently, in \cite{EFJS}, the JKO scheme was used by exploiting the variational structure induced by \eqref{freeenergy}. This approach yields universal bounds. However, arguments based on the free energy require additional constraints on the initial data. By contrast, the argument we present here relies only on the $L^{1}$-norm of the initial data. Unlike the present \(L^1(\mathbb R^2)\) setting, where
$t^{1/4}\|e^{t\Delta}u_0\|_{L^{4/3}}\to0
\quad\text{as }t\to0$,
the measure-valued framework considered in \cite{BM} allows atomic initial data, for which the standard perturbative argument fails and must be replaced by a decomposition around nonlinear self-similar profiles.
 
    The local well-posedness of the Keller-Segel system \eqref{KS} in $L^1$ can be found in \cite{W}. The main result of this section is the Theorem below, in which we add some additional properties of the flow such as regularization results. The proof follows the by-now classical arguments as initiated for the semilinear heat equation by Weissler \cite{Weissler1,Weissler2}. Theorem \ref{th:lwp:L1} serves in the present article as a unifying framework for several previously known results, namely \cite{N1}, \cite{M}, and \cite{BC}.

	\begin{theorem}\label{th:lwp:L1}
		
		Let $u_0\in L^1(\mathbb R^2)$. Then there exists a maximal time of existence $T=T(u_0)$ and a mild solution $u$ to the Keller--Segel system \eqref{KS} on $[0,T)$ such that
		
		\begin{itemize}
			\item \emph{Continuity of the flow in $L^1$.} We have $u\in C([0,T(u_0)),L^1(\mathbb R^2))$, and the conservation of mass \eqref{id:mass-conservation}. Moreover, the flow is locally continuous in the sense that for any $\delta>0$ and $0<T'<T$, there exists $\varepsilon>0$ (depending on $u_0$, $\delta$, and $T'$) such that if $\|v_0-u_0\|_{L^1}\leq \varepsilon$, then $T(v_0)\geq T'$ and
			\[
			\|v(t)-u(t)\|_{L^1}\leq \delta
			\]
			for all $t\in[0,T']$.
			
			\item \emph{Instantaneous regularization.} We have
			$u,\nabla\Phi_u\in C^\infty((0,T)\times\mathbb R^2)$
			and $u$ is a classical solution to \eqref{KS} on $(0,T)$. Moreover, for every $k\in\mathbb N$,
			$u\in C((0,T),W^{k,\infty}\cap W^{k,1}(\mathbb R^2))$.
			
			\item \emph{Blow-up criterion.} We have
			\[
			T(u_0)<\infty
			\quad\Longleftrightarrow\quad
			\lim_{t\uparrow T}\|u(t)\|_{L^\infty(\R^{2})}=\infty .
			\]
			Moreover, if $T(u_0)<\infty$, then
			\begin{align*}
				\|u(t)\|_{L^{\infty}(\R^{2})}\gtrsim \frac{1}{T-t}.
			\end{align*}
			\item \emph{Uniqueness.} The solution is unique in the class of mild solutions. More precisely, if $u'$ is another mild solution on $[0,T')$ with the same initial datum $u_0$, then $T'\leq T$ and $u=u'$ on $[0,T')$.
		\end{itemize}
		
	\end{theorem}

\begin{remark}

\noindent \emph{Improvement of the blow-up criterion}. We emphasize that the blow-up criterion proved here has been improved in the radial setting. More precisely, it is known (see \cite{S} and \cite{M} for further details) that
\begin{align}\label{typeIIblow-up}
	\limsup_{t\to T}(T-t)\|u(t)\|_{L^{\infty}(\R^{2})}=\infty.
\end{align}
Moreover, in \cite{NaiSu} the same result was established for bounded domains without assuming radial symmetry of the initial data but under the finite entropy condition. In \cite{SBook}, the authors prove \eqref{typeIIblow-up} under the additional assumption that the initial data have finite second moment. We believe it holds true without the second momentum assumption.

\smallskip

\noindent \emph{Comparison with the quadratic heat equation}. The Keller-Segel equation \eqref{KS} can be rewritten as a non-local evolution equation
$$
u_t= \Delta u +u^2+((2\pi)^{-1}\frac{x}{|x|^2}*u)\cdot \nabla u.
$$
However, the conclusions of Theorem \ref{th:lwp:L1} fail for the two-dimensional quadratic equation $u_t= \Delta u +u^2$, see Section 7.5. in \cite{BC} where this "doubly critical" case is investigated. The key point behind the validity of Theorem \ref{th:lwp:L1} for the Keller-Segel equation is the divergence structure of the nonlinearity.

\end{remark}

Using the Hardy--Littlewood--Sobolev, Young, and H\"older inequalities, one verifies that for a mild solution \ref{def:solution} the map
\[
s\mapsto \nabla G(t-s)*\cdot\bigl(u(s)\nabla\Phi_{u(s)}\bigr)
\]
belongs to $L^1([0,t]\times\mathbb R^2)$ for every $t\in[0,T)$. Hence the integral term in \eqref{id:mild-solution} is well defined and the identity holds in the sense of measurable functions. For $4/3\leq p<2$, if
\[
u\in C([0,T),L^p(\mathbb R^2)),
\]
then
\[
s\mapsto \nabla G(t-s)*\cdot\bigl(u(s)\nabla\Phi_{u(s)}\bigr)
\in L^1([0,t],L^p(\mathbb R^2)),
\]
so that \eqref{id:mild-solution} is also well defined as an identity in $L^p(\mathbb R^2)$. Throughout this section, as in all articles of the literature, the Poisson field is defined by
\begin{equation}\label{gradient}
	\nabla \Phi_u(x):=
	-\frac{1}{2\pi}
	\int_{\mathbb R^2}
	\frac{x-y}{|x-y|^2}u(y)\,dy.
\end{equation}
This definition is meaningful for every $u\in L^1(\mathbb R^2)$ (as the integral is absolutely convergent for almost every $x$) and does not require any additional decay assumptions at infinity. The next Lemma establishes the basic properties of the Poisson field that will be used repeatedly in the sequel.

\begin{lemma}\label{NonLinLpEst}
	
	Let $1<p<2$ and $1<q\le \infty$ satisfy
	\begin{equation}
		\frac1p+\frac1q\le\frac32.
	\end{equation}
	Then there exists $C>0$ such that for every $u\in L^p(\mathbb R^2)$, $v\in L^q(\mathbb R^2)$, and every $1\le r\le\infty$ satisfying
	\begin{equation}
		\frac1r\le \frac1p+\frac1q-\frac12,
	\end{equation}
	there holds
	\begin{equation}\label{nonlinestimate}
	\big\| \nabla G(t) \ast ( v \nabla \Phi_{u} ) \big\|_{L^r(\mathbb{R}^2)} 
\le \frac{C}{t^{\frac{1}{p}+\frac{1}{q}-\frac{1}{r}}} \, \|u\|_{L^p(\mathbb{R}^2)} \, \|v\|_{L^q(\mathbb{R}^2)}.
	\end{equation}
	
\end{lemma}

\begin{proof}
	By the Hardy--Littlewood--Sobolev inequality,
	\begin{equation}
		\|\nabla\Phi_u\|_{L^a(\mathbb R^2)}
		\le C \|u\|_{L^p(\mathbb R^2)},
		\qquad
		\frac1a=\frac1p-\frac12.
	\end{equation}
	Defining $b$ by
		$\frac1b=\frac1a+\frac1q$,
	H\"older's inequality yields
	\begin{equation}
		\|v\nabla\Phi_u\|_{L^b(\mathbb R^2)}
		\le
		C
		\|u\|_{L^p(\mathbb R^2)}
		\|v\|_{L^q(\mathbb R^2)}.
	\end{equation}
	Since $\frac1r\le\frac1b$, there exists $w\in[1,\infty]$ such that
		$1+\frac1r=\frac1b+\frac1w$.
	Young's inequality and the estimate
	\begin{align*}
		\|\nabla G(t)\|_{L^w(\mathbb R^2)}
		\le
		Ct^{-\frac32+\frac1w}
	\end{align*}
	yield \eqref{nonlinestimate}.
\end{proof}

\subsection{Local well-posedness in $L^p$ for $4/3\le p<2$}

We start with the local well-posedness in the supercritical space $L^p$ for $4/3\le p <2$ as this case is easier than for $L^1$.

The lower bound $p\ge 4/3$ is precisely the condition that allows us to apply Lemma~\ref{NonLinLpEst} with $q=r=p$, which is the key estimate in the fixed-point argument below.

\begin{theorem}\label{ExistenceThm}
	
	Let $\frac43\le p<2$. Then for any $u_0\in L^p(\mathbb R^2)$ there exists
	$T^*=T^*(\|u_0\|_{L^p})
	=
	\delta \|u_0\|_{L^p}^{-\frac{p}{p-1}}$
	for some $\delta>0$, and a unique mild solution
	$u\in C([0,T^*],L^p(\mathbb R^2))$
	to \eqref{KS}. Moreover, $\|u(t)\|_{L^p(\mathbb R^2)}
	\le2\|u_0\|_{L^p(\mathbb R^2)}$
	for all $0\le t\le T^*$.
	
	As a consequence, there exists a unique maximal mild solution in
	$C((0,T),L^p(\mathbb R^2))$. Furthermore, for some universal constant
	$C>0$,
	\[
	T<\infty
	\qquad\Longleftrightarrow\qquad
	\lim_{t\uparrow T}\|u(t)\|_{L^p(\mathbb R^2)}=\infty.
	\]
	Moreover, if $T<\infty$, then
	\[
	\|u(t)\|_{L^p(\mathbb R^2)}
	\ge
	\frac{C}{(T-t)^{1-\frac1p}}
	\]
	for all $t<T$.
	
\end{theorem}

\begin{proof}
	Let
	$X_T:=C([0,T],L^p(\mathbb R^2))$,
	with
	\[
	\|v\|_{X_T}
	:=
	\sup_{0\le t\le T}
	\|v(t)\|_{L^p(\mathbb R^2)}.
	\]
	Define
	\[
	F(v)(t)
	:=
	G(t)u_0
	-
	\int_0^t
	\nabla G(t-s)\ast\cdot
	\bigl(v(s)\nabla\Phi_{v(s)}\bigr)\,ds.
	\]
	Since $p\ge 4/3$, Lemma~\ref{NonLinLpEst} applies with $q=r=p$. Hence
	\[
	\|F(v)\|_{X_T}
	\le
	\|u_0\|_{L^p(\mathbb R^2)}
	+
	C T^{1-\frac1p}
	\|v\|_{X_T}^2,
	\]
	and
	\[
	\|F(v)-F(u)\|_{X_T}
	\le
	C T^{1-\frac1p}
	\bigl(\|u\|_{X_T}+\|v\|_{X_T}\bigr)
	\|u-v\|_{X_T}.
	\]
    The continuity $F(v)\in C([0,T),L^p)$ follows from similar estimates, since this is standard we refer to \cite{BC}. Therefore, choosing
	\[
	T
	=
	\delta
	\|u_0\|_{L^p(\mathbb R^2)}^{-\frac{p}{p-1}}
	\]
	with $\delta>0$ sufficiently small, $F$ is a contraction on
	$B_{X_T}(2\|u_0\|_{L^p(\mathbb R^2)})$.
	By Banach's fixed-point theorem, there exists a unique mild solution
	$u\in C([0,T],L^p(\mathbb R^2))$ satisfying
	\[
	\|u(t)\|_{L^p(\mathbb R^2)}
	\le
	2\|u_0\|_{L^p(\mathbb R^2)}
	\qquad
	0\le t\le T.
	\]
	
	Restarting the argument at any time $t_0<T$ gives the maximal mild solution and the continuation criterion. If $T<\infty$, then
	$T-t_0
	\ge
	\delta
	\|u(t_0)\|_{L^p(\mathbb R^2)}^{-\frac{p}{p-1}}$,
	hence
	\[
	\|u(t_0)\|_{L^p(\mathbb R^2)}
	\ge
	\frac{C}{(T-t_0)^{1-\frac1p}}.
	\]
	Conversely, if
	$\sup_{t<T}\|u(t)\|_{L^p(\mathbb R^2)}<\infty$,
	then the same local theory extends the solution beyond $T$. This proves the blow-up criterion.
\end{proof}

\subsection{Parabolic regularization}\label{RegularizationSection}

The next proposition provides a first regularization mechanism for mild solutions. It follows the standard approach for semilinear parabolic equation which can be found in \cite{BC}. Starting from an $L^p$ solution with $4/3\le p<2$, any additional $L^q$ control immediately propagates to higher Lebesgue spaces.

\begin{proposition}\label{AddRegProp}
	
	Let $u\in C([0,T],L^p(\mathbb R^2))$ be a mild solution for some
	$\frac43\le p<2$, and suppose moreover that
	\[
	u\in C([0,T],L^q(\mathbb R^2))
	\]
	for some $q\ge p$. Then the following regularity properties hold:
	\begin{itemize}
		\item if $\frac1p+\frac1q\le 1$, then
		\[
		u\in C((0,T],L^r(\mathbb R^2))
		\]
		for every $q\le r\le \infty$;
		
		\item if $\frac1p+\frac1q>1$, then
		\[
		u\in C((0,T],L^r(\mathbb R^2))
		\]
		for every
		\[
		q\le r<
		\Big(\frac1p+\frac1q-1\Big)^{-1}.
		\]
	\end{itemize}
		Moreover, in each case, for any $0<T_0<T$ there exists $C(T_0,T,\sup_{0\leq t\leq T}\| u(t)\|_{L^p\cap L^q})$ which is non-decreasing with the third variable such that $\| u(t)\|_{L^r}\leq C(T_0,T,\sup_{0\leq t\leq T}\| u(t)\|_{L^p\cap L^q})$ for all $T_0\leq t \leq T$.
	
\end{proposition}

\begin{proof}
	
	We write
	$u=v+w$,
	where
	\[
	v(t)=G(t)u_0,
	\qquad
	w(t)
	=
	-\int_0^t
	\nabla G(t-s)\ast\cdot
	\bigl(u(s)\nabla\Phi_{u(s)}\bigr)\,ds.
	\]
	By the standard smoothing properties of the heat semigroup
	$v\in C((0,T],L^r(\mathbb R^2))$
	for every $r\in[q,\infty]$, and for every $T_0>0$,
	$\sup_{T_0\le t\le T}
	\|v(t)\|_{L^r(\mathbb R^2)}
	<\infty$. Let
	$M=
	\sup_{0\le t\le T}
	\|u(t)\|_{L^p(\mathbb R^2)\cap L^q(\mathbb R^2)}$.
	Applying Lemma~\ref{NonLinLpEst}, we obtain
	\[
	\|w(t)\|_{L^r(\mathbb R^2)}
	\le
	CM^2
	\int_0^t
	(t-s)^{-\gamma}\,ds,
	\qquad
	\gamma
	=
	\frac1p+\frac1q-\frac1r.
	\]
	The integral is finite precisely when $\gamma<1$, that is,
	$\frac1p+\frac1q-\frac1r<1$.
	In this case,
	\[
	\sup_{0\le t\le T}
	\|w(t)\|_{L^r(\mathbb R^2)}
	\le C(M,T).
	\]
	Moreover, arguing exactly as in the proof of Theorem~\ref{ExistenceThm}, the continuity of the Duhamel term follows from the integrability of the kernel $(t-s)^{-\gamma}$. Hence
	$w\in C([0,T],L^r(\mathbb R^2))$.
	
	Combining the regularity of $v$ and $w$ this yields the desired result. The final uniform bound on $[T_0,T]$ follows from the preceding estimates.
	
\end{proof}
\begin{corollary}\label{ImprovedLinf}
	Let $\frac{4}{3}\le p<2$ and let $u$ be the mild solution in
	$C([0,T),L^p(\mathbb R^2))$ given by Theorem~\ref{ExistenceThm}. Then
	\[
	u\in C((0,T),L^\infty(\mathbb R^2)).
	\]
	In addition, if $T<\infty$, then
	\[
	\|u(t)\|_{L^\infty(\mathbb R^2)}
	\ge
	\frac{C}{T-t}
	\]
	for all $t<T$. Moreover, with $T^*$ as in Theorem~\ref{ExistenceThm}, for every
	$\epsilon\in(0,1)$ and every $R>0$, if
	$\|u_0\|_{L^p(\mathbb R^2)}\le R$, then
	\[
	\|u(t)\|_{L^p(\mathbb R^2)\cap L^\infty(\mathbb R^2)}
	\le
	C(\epsilon,R)
	\]
	for all $t\in[\epsilon T^*(R),T^*(R)]$.
\end{corollary}

\begin{proof}
	By Proposition~\ref{AddRegProp}, starting from
	$u\in C([0,T),L^p(\mathbb R^2))$, we obtain successively
	$u\in C((0,T),L^q(\mathbb R^2))$ for larger and larger exponents $q$.
	Indeed, if $\frac1p+\frac1q>1$,
	then Proposition~\ref{AddRegProp} gives
	$u\in C((0,T),L^{q'}(\mathbb R^2))$
	for every $q<q'<\Big(\frac1p+\frac1q-1\Big)^{-1}$.
	Since $q'>q$, after finitely many iterations we reach an exponent
	$q\ge \frac{p}{p-1}$, so that $\frac1p+\frac1q\le 1$. A final application of
	Proposition~\ref{AddRegProp} yields
	\[
	u\in C((0,T),L^\infty(\mathbb R^2)).
	\]
	
	We now prove the blow-up lower bound. Multiplying the equation by
	$u^{p-1}$ and integrating by parts, using $-\Delta\Phi_u=u$, gives
	\[
	\frac{d}{dt}\|u(t)\|_{L^p(\mathbb R^2)}^p
	\le
	C\|u(t)\|_{L^\infty(\mathbb R^2)}
	\|u(t)\|_{L^p(\mathbb R^2)}^p.
	\]
	By approximation, this estimate remains valid for the above mild solutions.
	Thus, by Gronwall's inequality and the blow-up criterion of
	Theorem~\ref{ExistenceThm}, if $T<\infty$, then
	$\lim_{t\uparrow T}\|u(t)\|_{L^\infty(\mathbb R^2)}=\infty$.
	Fix $t_0<T$. The comparison principle applied to
	$v_t=\Delta v+v^2-\nabla v\cdot\nabla\Phi_u$ yields
	\[
	u(t,x)
	\le
	\Big(\|u(t_0)\|_{L^\infty(\mathbb R^2)}^{-1}-(t-t_0)\Big)^{-1}
	\]
	for
	$t<t_0+\|u(t_0)\|_{L^\infty(\mathbb R^2)}^{-1}$.
	Since $\|u(t)\|_{L^\infty}$ blows up as $t\uparrow T$, we obtain
	\[
	T-t_0
	\ge
	\|u(t_0)\|_{L^\infty(\mathbb R^2)}^{-1},
	\]
	which proves the desired bound.
	
	It remains to prove the short-time regularization estimate. If
	$\|u_0\|_{L^p(\mathbb R^2)}\le 1$, then Theorem~\ref{ExistenceThm} gives
	\[
	\sup_{0\le t\le T^*(1)}
	\|u(t)\|_{L^p(\mathbb R^2)}
	\le 2.
	\]
	The bootstrap argument above, together with Proposition~\ref{AddRegProp}, then implies that for every
	$\epsilon\in(0,1)$ there exists $C(\epsilon)>0$ such that
	\[
	\|u(t)\|_{L^p(\mathbb R^2)\cap L^\infty(\mathbb R^2)}
	\le
	C(\epsilon)
	\]
	for all $t\in[\epsilon T^*(1),T^*(1)]$. For general data with $\|u_0\|_{L^p(\mathbb R^2)}\le R$, set
	$\lambda=R^{-\frac{p}{2p-2}}$
	and consider the rescaled solution $u_\lambda(t,x)=\lambda^2u(\lambda^2t,\lambda x)$.
	Then $\|u_\lambda(0)\|_{L^p(\mathbb R^2)}\le 1$. Applying the previous estimate to
	$u_\lambda$ and scaling back gives
	\[
	\|u(t)\|_{L^p(\mathbb R^2)\cap L^\infty(\mathbb R^2)}
	\le
	C(\epsilon,R)
	\]
	for all $t\in[\epsilon T^*(R),T^*(R)]$.
\end{proof}
\begin{proposition}\label{SobolevspaceProp}
	
	Let $\frac43\le p<2$ and let $u$ be the mild solution in
	$C([0,T),L^p(\mathbb R^2))$ given by Theorem~\ref{ExistenceThm}. Then
	\[
	u\in C((0,T),W^{k,\infty}(\mathbb R^2)\cap W^{k,p}(\mathbb R^2))
	\]
	for every $k\in\mathbb N$. \newline Moreover, with $T^*$ as in Theorem~\ref{ExistenceThm}, for every
	$\epsilon\in(0,1)$ and every $R>0$, if
	$\|u_0\|_{L^p(\mathbb R^2)}\le R$,
	then
	\[
	\|u(t)\|_{W^{k,\infty}(\mathbb R^2)\cap W^{k,p}(\mathbb R^2)}
	\le
	C(\epsilon,R,k)
	\]
	for all
	\[
	t\in[\epsilon T^*(R),T^*(R)].
	\]
	
\end{proposition}

\begin{proof}
	
	The case $k=0$ is precisely Corollary~\ref{ImprovedLinf}. Fix
	$t_0\in(0,T)$. By Corollary~\ref{ImprovedLinf},
	$u(t_0)\in L^p(\mathbb R^2)\cap L^\infty(\mathbb R^2)$,
	and
	\[
	\sup_{t_0\le t\le T}
	\|u(t)\|_{L^p(\mathbb R^2)\cap L^\infty(\mathbb R^2)}
	<\infty.
	\]
	
	Restarting the equation at time $t_0$, differentiating the Duhamel formula and applying Lemma~\ref{NonLinLpEst}
	to the terms
	$\nabla u\,\nabla\Phi_u,$ and $u\,\nabla\Phi_{\nabla u}$,
	one obtains
	\[
	\sup_{0<t\le T_0}
	\sqrt t\,
	\|\nabla u(t+t_0)\|_{L^p(\mathbb R^2)\cap L^\infty(\mathbb R^2)}
	<\infty
	\]
	for some $T_0=T_0(\|u(t_0)\|_{L^p\cap L^\infty})>0$ (performing a fixed point argument in the space corresponding to the above norm, exactly as in the proof of Theorem \ref{ExistenceThm}).
	By uniqueness, this estimate applies to the original solution and therefore
	$u\in
	C((0,T),W^{1,p}(\mathbb R^2)\cap W^{1,\infty}(\mathbb R^2))$. Repeating the argument for higher derivatives yields inductively
	\[
	u\in
	C((0,T),W^{k,p}(\mathbb R^2)\cap W^{k,\infty}(\mathbb R^2))
	\]
	for every $k\in\mathbb N$.
	
	Finally, the quantitative estimate on
	$[\epsilon T^*(R),T^*(R)]$ follows by combining the uniform
	$L^p\cap L^\infty$ bounds of Corollary~\ref{ImprovedLinf}
	with the above bootstrap argument. The constants depend only on
	$\epsilon$, $R$ and $k$.
	
\end{proof}
	We observe that as a consequence the Proposition \ref{SobolevspaceProp} the solutions given by Theorem \ref{ExistenceThm} satisfy
	\begin{align*}
		u(t)=G(t)u_{0}+\int_{0}^{t}G(t-s)\nabla \cdot (u\Phi_{u})(s)ds.
	\end{align*}

\begin{corollary}\label{cor:Cinfty-regularity}
	
	Let $\frac43\le p<2$ and let $u$ be the mild solution in
	$C([0,T),L^p(\mathbb R^2))$ given by Theorem~\ref{ExistenceThm}. Then
	\[
	u\in C^\infty((0,T)\times\mathbb R^2).
	\]
	Moreover, for every $\epsilon\in(0,1)$ and every $k\in\mathbb N$, recalling
	$T^*$ from Theorem~\ref{ExistenceThm}, if $\|u_0\|_{L^p(\mathbb R^2)}
	\le R$,
	then
	\[
	\|\partial_t^\alpha \nabla^{\alpha'}u(t)\|_{L^p(\mathbb R^2)\cap L^\infty(\mathbb R^2)}
	\le
	C(k,\epsilon,R)
	\]
	for every $\alpha+\alpha'\le k$ and every
	\[
	t\in[\epsilon T^*(R),T^*(R)].
	\]
	
\end{corollary}

\begin{proof}
	
	By Proposition~\ref{SobolevspaceProp},
	\[
	u\in C((0,T),W^{m,p}(\mathbb R^2)\cap W^{m,\infty}(\mathbb R^2))
	\]
	for every $m\in\mathbb N$. In particular, Sobolev embedding implies that
	$u$ is smooth with respect to the spatial variables.
	
	Next, using \eqref{KS}
	and the estimate
	\[
	\|\nabla^{m+1}\Phi_u\|_{L^p(\mathbb R^2)\cap L^\infty(\mathbb R^2)}
	\le
	C
	\|\nabla^m u\|_{L^p(\mathbb R^2)\cap L^\infty(\mathbb R^2)},
	\]
	which follows from the representation formula \eqref{gradient} for $\Phi_u$, we obtain
	\[
	\|u_t\|_{W^{m,p}(\mathbb R^2)\cap W^{m,\infty}(\mathbb R^2)}
	\le
	C_m
	\Big(
	1+
	\|u\|_{W^{m+2,p}(\mathbb R^2)\cap W^{m+2,\infty}(\mathbb R^2)}^2
	\Big).
	\]
	Hence all first-order time derivatives belong to
	$W^{m,p}\cap W^{m,\infty}$.
	
	Differentiating the equation repeatedly in time shows inductively that every
	$\partial_t^\alpha u$ can be expressed as a finite combination of products of
	spatial derivatives of $u$ and $\Phi_u$. Since all spatial derivatives are
	already controlled by Proposition~\ref{SobolevspaceProp}, it follows that
	\[
	u\in C^\infty((0,T)\times\mathbb R^2).
	\]
	
	Finally, the quantitative estimate on
	$[\epsilon T^*(R),T^*(R)]$ follows immediately from the uniform bounds of
	Proposition~\ref{SobolevspaceProp} and the above induction argument.
	
\end{proof}

\subsection{Local well-posedness in $L^1(\mathbb R^2)$}\label{subsec:lwp-L1}

In this subsection we establish local well-posedness in the critical space
$L^1(\mathbb R^2)$. The argument follows the classical approach of
Brezis and Cazenave \cite{BC}.

\begin{lemma}\label{BrezisCazenave}
	For every $f\in L^1(\mathbb R^2)$,
	\[
	\lim_{t\downarrow0}
	t^{1/4}\|G(t)\ast f\|_{L^{4/3}(\mathbb R^2)}
	=0.
	\]
\end{lemma}

\begin{proof}
	The claim is immediate for $f\in\mathcal S(\mathbb R^2)$ since
	\[
	t^{1/4}\|G(t)\ast f\|_{L^{4/3}(\mathbb R^2)}
	\le
	t^{1/4}\|f\|_{L^{4/3}(\mathbb R^2)}.
	\]
	The general case follows by approximation in $L^1(\mathbb R^2)$ and the estimate
	\[
	\|G(t)\ast g\|_{L^{4/3}(\mathbb R^2)}
	\le
	Ct^{-1/4}\|g\|_{L^1(\mathbb R^2)},
	\]
	which is a consequence of Young's inequality.
\end{proof}

\begin{theorem}\label{ExistenceL1}
	
	For every $u_0\in L^1(\mathbb R^2)$ there exists $T_0>0$ and a unique mild solution in the sense of Definition~\ref{def:solution}
	of \eqref{KS} on $[0,T_{0}]$. Moreover, for every $\delta>0$ there exists $\varepsilon>0$ such that if
	\[
	\|v_0-u_0\|_{L^1(\mathbb R^2)}
	\le
	\varepsilon,
	\]
	then the corresponding solution $v$ exists on $[0,T_0]$ and satisfies
	\[
	\sup_{0\le t\le T_0}
	\|u(t)-v(t)\|_{L^1(\mathbb R^2)}
	\le
	\delta.
	\]
	
\end{theorem}

\begin{proof}
	
	As the reasoning is classical, we shall refer to \cite{BC} for the details. Let
	\[
	X=
	C([0,T],L^1(\mathbb R^2))
	\cap
	\Big\{
	u:
	t^{1/4}u(t)\in C([0,T],L^{4/3}(\mathbb R^2)),
	\;
	\lim_{t\downarrow0}
	t^{1/4}\|u(t)\|_{L^{4/3}(\mathbb R^2)}=0
	\Big\},
	\]
	endowed with the norm
	$\|u\|_X
	=
	\sup_{0\le t\le T}
	\|u(t)\|_{L^1(\mathbb R^2)}
	+
	\sup_{0\le t\le T}
	t^{1/4}\|u(t)\|_{L^{4/3}(\mathbb R^2)}$.
    
	Consider the map
	\[
	F(v)(t)
	=
	G(t)u_0
	-
	\int_0^t
	\nabla G(t-s)\ast\cdot
	\bigl(v(s)\nabla\Phi_{v(s)}\bigr)\,ds.
	\]
	For $\nu>0$ let
	\[
	E=
	\Bigl\{
	u\in X:
	\sup_{0\le t\le T}
	\|u(t)\|_{L^1(\mathbb R^2)}
	\le
	2\|u_0\|_{L^1(\mathbb R^2)},
	\quad
	\sup_{0\le t\le T}
	t^{1/4}\|u(t)\|_{L^{4/3}(\mathbb R^2)}
	\le
	\nu
	\Bigr\}.
	\]
	By Lemma~\ref{BrezisCazenave}, choosing $T$ sufficiently small gives
	\[
	\sup_{0\le t\le T}
	t^{1/4}
	\|G(t)u_0\|_{L^{4/3}(\mathbb R^2)}
	\le
	\frac{\nu}{2}.
	\]
	Moreover, applying Lemma~\ref{NonLinLpEst} with $p=q=r=4/3$ yields
	\[
	\Big\|
	\int_0^t
	\nabla G(t-s)\ast\cdot
	(v\nabla\Phi_v)(s)\,ds
	\Big\|_{L^1(\mathbb R^2)}
	+
	t^{1/4}
	\Big\|
	\int_0^t
	\nabla G(t-s)\ast\cdot
	(v\nabla\Phi_v)(s)\,ds
	\Big\|_{L^{4/3}(\mathbb R^2)}
	\le
	C\nu^2.
	\]
	Hence $F$ maps $E$ into itself provided $\nu$ and then $T$ are sufficiently small.
	
	An analogous estimate
	shows that $F$ is a contraction on $E$. Therefore Banach's fixed-point theorem yields a unique solution
	$u\in E$.
	
	The continuous dependence statement is obtained by applying the same contraction argument to the difference
	$w=v-u$.
	Choosing $\varepsilon$ sufficiently small, one obtains
	\[
	\|w\|_X\le \delta,
	\]
	which implies the desired estimate.
	
\end{proof}

We can now finish the proof of the main result of this section, by gathering all the previous intermediate results.

\begin{proof}[Proof of Theorem \ref{th:lwp:L1}]
	
	The existence, uniqueness, and continuous dependence in $L^1(\mathbb R^2)$ follow from Theorem~\ref{ExistenceL1}, together with the standard continuation argument yielding a unique maximal mild solution on a maximal interval $[0,T)$.
	
	The conservation of mass follows from the identity
	\[
	\frac{d}{dt}\int_{\mathbb R^2}u(t,x)\,dx=0,
	\]
	which holds for smooth and localized solutions, and then by a standard approximation argument.
	
	Since
	$u(t_0)\in L^{4/3}(\mathbb R^2)$
	for every $t_0>0$, uniqueness implies that the restriction of $u$ to $[t_0,T)$ coincides with the mild solution provided by Theorem~\ref{ExistenceThm} with initial datum $u(t_0)$. Proposition~\ref{SobolevspaceProp} therefore yields
	$u\in C((0,T),W^{k,4/3}(\mathbb R^2)\cap W^{k,\infty}(\mathbb R^2))$
	for every $k\in\mathbb N$. Moreover, a straightforward adaptation of the proof of Proposition~\ref{SobolevspaceProp}, using Lemma~\ref{NonLinLpEst} with $p=q=4/3$ and $r=1$, gives
	$u\in C((0,T),W^{k,1}(\mathbb R^2))$
	for every $k\in\mathbb N$. In particular,
	$u$, $\nabla\Phi_u\in C^\infty((0,T)\times\mathbb R^2)$,
	and $u$ is a classical solution of \eqref{KS} on $(0,T)$.
	
	Finally, the blow-up criterion follows from Corollary~\ref{ImprovedLinf}.
	
\end{proof}

\section{Local well-posedness in $L^1(\langle x \rangle^2 dx)$}\label{sec:lwp-L1x2}

\subsection{Local well-posedness in $L^1(\langle x \rangle^2 dx)$}

The main classification result of this article, Theorem \ref{thm:classification}, concerns initial data in $L^1(\langle x \rangle^2 dx)$. The result below states that the Keller-Segel system \eqref{KS} is well-posed in this space, where it preserves mass, center of mass and second momentum, and where by parabolic regularization the entropy is instantaneously well-defined and dissipated. As we did not see it clearly written down in the literature despite the tremendous use of this function space, we thought it deserved a statement and proof in the present article.
For the sake of completeness, we recall that the free energy is well-defined under the following condition.

\begin{theorem} \label{thm:lwp-L1x2}
	
	For any nonnegative $u_0\in L^1(\langle x \rangle^2 dx)$, let $u$ be the solution to the Keller-Segel system \eqref{KS} provided by Theorem \ref{th:lwp:L1}. Then the following properties hold
	
	\begin{itemize}
		\item \emph{Continuity of the flow in $L^1(\langle x \rangle^2dx)$.} We have $u\in C([0,T(u_0)),L^1(\langle x \rangle^2 dx))$.
		\item \emph{Instantaneous boundedness and decrease of the free energy functional}. For all $t\in (0,T(u_0))$ we have $u(t)|\log u(t)|,u(t)|\Phi_{u(t)}|\in L^1(\mathbb R^2)$, so that the free energy $\mathcal F[u](t)$ \eqref{freeenergy} is well-defined. We also have that $u|\nabla \log u-\nabla \Phi_u|^2\in L^1_{loc}((0,T),L^1(\mathbb R^2))$ and the identity \eqref{freeenergy-2} holds true for all $0<t<t'<T(u_0)$.
		\item \emph{First and second momenta}. There hold the identities \eqref{first-momentum} and \eqref{second-momentum}.
	\end{itemize}

\end{theorem}
\begin{proof}
		
		\noindent \underline{Continuity of the flow map in $L^1(\langle x \rangle^2 dx)$}.This result follows from the fact that the Cauchy problem associated with the Keller--Segel system \eqref{KS} is well-posed in the weighted space $L^1(\langle x \rangle^2 \, dx)$. Since the underlying arguments are essentially analogous to those used in establishing well-posedness in $L^1$, see for instance \cite{BC}, we shall confine ourselves to a sketch of the principal ideas.  
		
		First, for all $1 \leq p \leq q \leq \infty$, one has the localized heat kernel estimate
		\begin{equation} \label{bd:heat-kernel-localized}
			\| \nabla G(t) * f \|_{L^q(\langle x \rangle^2 dx)} 
			\lesssim \frac{1}{t^{\frac{1}{2} + \frac{1}{p} - \frac{1}{q}}} 
			\| f \|_{L^p(\langle x \rangle^2 dx)},
		\end{equation}
		valid for $0 < t \leq 1$. This inequality is a consequence of the pointwise bound
		\begin{align*}
			|\langle x \rangle^{2} (\nabla G \ast f)(x)| 
			&\lesssim t \int_{\mathbb{R}^{2}} \frac{|x-y|}{t} \, \frac{|x-y|^{2}}{t} 
			|G(x-y,t)| \, |f(y)| \, dy \\
			&\quad + \int_{\mathbb{R}^{2}} \frac{|x-y|}{t} \, |G(x-y,t)| \, |f(y)| \, |y|^{2} \, dy .
		\end{align*}
		It is clear that the first contribution on the right-hand side represents only a perturbative term with respect to \eqref{bd:heat-kernel-localized}.  
		
		Second, we decompose
		\[
		\nabla \Phi_f 
		= -\frac{1}{2\pi}\Big(\frac{\cdot}{|\cdot|^2}\chi (\cdot)\Big)*f
		-\frac{1}{2\pi}\Big(\frac{\cdot}{|\cdot|^2}(1-\chi (\cdot))\Big)*f
		= \Theta_f+\Theta'_f,
		\]
		where for $1<q<2$, each component satisfies the estimates
		\begin{equation} \label{bd:Hardy-Littlewood-Sobolev-localized}
			\left\| \Theta_f \right\|_{L^{\frac{2q}{2-q}}(\langle x \rangle^2 dx)}\lesssim \| f\|_{L^q(\langle x \rangle^2 dx)},
			\qquad 
			\|\Theta'_f\|_{L^\infty}\lesssim \| f\|_{L^1}.
		\end{equation}
		The first inequality follows from Hardy--Littlewood--Sobolev inequality, after noting that
		\begin{align*}
			|\Theta_{f}(x)|
			&\lesssim \int_{\R^2}\frac{1}{|x-y|}\chi(|x-y|)|f(y)|\,dy \\
			&\lesssim \int_{\R^{2}}\frac{1}{|x-y|}\chi(|x-y|)\langle y \rangle ^{2/q}|f(y)|\,dy,
		\end{align*}
		while the second estimate follows directly from the trivial bound $\langle y\rangle \geq 1$. The fixed-point formulation for mild solutions is then expressed as
		\begin{equation} \label{id:fixed-point-localized}
			u(t)=G(t)*u_0-\int_0^t \nabla G(t-s)* \big(u(s)\Theta_{u(s)}\big)\,ds
			-\int_0^t \nabla G(t-s)* \big(u(s)\Theta'_{u(s)}\big)\,ds .
		\end{equation}
		If $u\in C([0,T],L^1(\langle x \rangle^2 dx))$ with $t^{1/4}u\in L^{\infty}([0,T],L^{4/3}(\langle x \rangle^2 dx))$, then the second term in \eqref{id:fixed-point-localized} can be controlled by means of \eqref{bd:heat-kernel-localized} and \eqref{bd:Hardy-Littlewood-Sobolev-localized}. Indeed, observing that
		\[
		|u(s)\Theta_{u(s)}|\langle x \rangle^{2}
		=|u(s) \Theta_{u(s)}|\langle x \rangle^{2(\frac{1}{4/3}+\frac{1}{4})},
		\]
		we obtain
		\begin{align*}
			\| \nabla G(t-s)* (u(s)\Theta_{u(s)})\|_{L^1(\langle x \rangle^2)} 
			&\lesssim \frac{1}{(t-s)^{1/2}}\| u(s)\Theta_{u(s)}\|_{L^1(\langle x \rangle^2 dx)}\\
			&\lesssim \frac{1}{(t-s)^{1/2}} \| u(s)\|_{L^{4/3}(\langle x \rangle^2 dx)}\| \Theta_{u(s)}\|_{L^{4}(\langle x \rangle^2 dx)} \\
			&\lesssim \frac{1}{(t-s)^{1/2}s^{1/2}} \| t^{1/4}u\|_{L^{\infty}([0,T],L^{4/3}(\langle x \rangle^2 dx))}^2.
		\end{align*}
		Alternatively,
		\begin{align*}
			\|\nabla G(t-s)\ast (u(s)\Theta_{u(s)})\|_{L^{4/3}}
			&\lesssim \frac{1}{(t-s)^{1/2}}\|u(s)\Theta_{u(s)}\|_{L^{1}}\\
			&\lesssim \frac{1}{(t-s)^{1/2}}\|u(s)\|_{L^{4/3}}\|\Theta_{u(s)}\|_{L^{4}} \\
			&\lesssim \frac{1}{s^{1/2}(t-s)^{1/2}}\|t^{1/4}u(t)\|^{2}_{L^{\infty}([0,T],L^{4/3}(\langle x \rangle^{2} dx))}.
		\end{align*}
		
		Finally, the third term in \eqref{id:fixed-point-localized} can be estimated, for $p=1,4/3$, as
		\begin{align*}
			\Big\| \int_{0}^{t} \nabla G(t-s)* (u(s)\Theta'_{u(s)})\,ds\Big\|_{L^p(\langle x \rangle^2 dx)} 
			&\lesssim \int_{0}^{t} \frac{1}{(t-s)^{1/2}}\| u(s)\Theta'_{u(s)}\|_{L^1(\langle x \rangle^2 dx)}\,ds \\
			&\lesssim \int_{0}^{t}\frac{1}{(t-s)^{1/2}} \| u(s)\|_{L^{1}(\langle x \rangle^2 dx)}\| \Theta'_{u(s)}\|_{L^{\infty}}\,ds\\
			&\lesssim \int_{0}^{t}\frac{1}{(t-s)^{1/2}} \| u(s)\|^{2}_{L^{1}(\langle x \rangle^2 dx)}\,ds \\
			&\lesssim t \, \|u(t)\|^{2}_{L^{\infty}([0,T],L^{1}(\mathbb{R}^{2}))}.
		\end{align*}
		Thanks to Theorem \ref{th:lwp:L1}, one may choose $t$ sufficiently small so that this last term becomes arbitrarily small. These estimates constitute the key ingredients for solving \eqref{id:fixed-point-localized} via a fixed-point argument in the class of functions 
		\[
		u \in C([0,T],L^1(\langle x \rangle^2 dx)), 
		\qquad 
		t^{1/4}u \in L^{\infty}([0,T],L^{4/3}(\langle x \rangle^2 dx)),
		\]
		with the additional condition 
		\[
		\lim_{t\downarrow 0} t^{1/4}\|u\|_{L^{\infty}([0,T],L^{4/3}(\langle x \rangle^2 dx))}=0,
		\]
		exactly as in the proof of Theorem \ref{th:lwp:L1}.
		
		\medskip
		
		\noindent 
	\end{proof}

\begin{lemma} \label{lem:free-energy-well-def}
	If $u\geq 0$ satisfies $u\in L^\infty \cap L^1(\langle x \rangle^2dx)$ then we have $u|\log u|,u|\Phi_u|\in L^1$ with
	$$
	\| u |\log u|\|_{L^1}\lesssim \| u \|_{L^1(\langle x\rangle^2dx)}\langle \log \langle \| u\|_{L^\infty}\rangle\rangle \quad \mbox{and} \quad \| u\Phi_u\|_{L^1}\lesssim \| u\|_{L^1(\langle x\rangle^2dx)}(\| u\|_{L^1(\langle x\rangle^2dx)}+\| u\|_{L^\infty}).
	$$
	In particular, the free energy $\mathcal F[u]$ is well-defined.
\end{lemma}

\begin{proof}
	We introduce the change of variables $u = \langle x \rangle^{-2} f$, so that 
	$\|f\|_{L^{1}} = \|u\|_{L^{1}(\langle x \rangle^{2}dx)}$.
	With this notation, we decompose
	\begin{align*}
		|u \log u|
		&\lesssim \langle x \rangle^{-2} f |\log f| \, \mathbbm{1}(f \geq 1) 
		+ \langle x \rangle^{-2} f |\log f| \, \mathbbm{1}(f < 1) 
		+ \langle x \rangle^{-2} \langle \log \langle x \rangle \rangle f :=I+II+III.
	\end{align*}
	
	On the set $\{ f  \geq 1 \}$, we observe since $u$ is bounded that $|\log f| \lesssim \langle \log \langle x \rangle \rangle+\log \langle \| u\|_{L^\infty}\rangle$. Since $\langle x \rangle^{-2}\langle \log \langle x \rangle \rangle$ is bounded, we see $\| I+III\|_{L^1}\lesssim \| f\|_{L^1}\langle \log \langle \| u\|_{L^\infty}\rangle\rangle$. On the set $\{ f < 1 \}$, we have $f |\log f|\lesssim \sqrt{f}$. Since the functions $\sqrt{f}$ and $\langle x\rangle^{-2}$ belong to $L^2$, we deduce $\| II\|_{L^1}\lesssim \| f\|_{L^1}$ by H\"older. This shows $u |\log u| \in L^{1}$ with the desired estimate.
	
	Next, for $x\in \mathbb R^2$, we estimate
	\begin{align*}
		& \int |\log |x-y|| \, |u(y)| \, dy
		= \int_{|x-y|\leq 1} \cdots \, dy
		+ \int_{\substack{|x-y|>1 \\ |x|\le |y|+1}} \cdots \, dy
		+ \int_{\substack{|x-y|>1 \\ |y|\le |x|+1}} \cdots \, dy \\
		& \qquad \lesssim \|u\|_{L^\infty} 
		+ \int_{|x|\leq |y|} \langle \log \langle y \rangle \rangle\langle y \rangle^{-2} f(y) \, dy 
		+ \langle \log \langle x \rangle \rangle \|u\|_{L^1}  \, \lesssim \|u\|_{L^\infty} 
		+ \langle \log \langle x \rangle \rangle \|u\|_{L^1} 
		+ \|f\|_{L^1} .
	\end{align*}
	
	Therefore,
	$
	|\Phi_u(x)| \lesssim \langle \log \langle x \rangle \rangle \,
	\|u\|_{L^1(\langle x \rangle^{2} dx)\cap L^\infty}$,
	implying $u\Phi_u\in L^1$ with the desired estimate.

\end{proof}

In the next lemma, we also establish the fundamental identities 
\eqref{first-momentum}, \eqref{second-momentum}, and \eqref{freeenergy-2}.  

\begin{lemma}
	Let $u_0 \in L^1(\langle x \rangle^2 dx)$ be nonnegative, and let $u$ denote the solution to the Keller--Segel system \eqref{KS} given by Theorem~\ref{th:lwp:L1}. Then the following properties hold:
	\begin{itemize}
		\item \emph{Instantaneous boundedness and decay of the free energy functional.}  
		For every $t \in (0,T(u_0))$ we have 
		\[
		u(t)|\log u(t)|, \quad u(t)|\Phi_{u(t)}| \in L^1(\mathbb{R}^2),
		\]
		so that the free energy $\mathcal F[u](t)$ \eqref{freeenergy} is well-defined.  
		Moreover, 
		\[
		u|\nabla \log u - \nabla \Phi_u|^2 \in L^1_{\mathrm{loc}}\big((0,T),L^1(\mathbb{R}^2)\big),
		\]
		and the identity \eqref{freeenergy-2} holds for all $0<t<t'<T(u_0)$.
		
		\item \emph{First and second moments.}  
		The identities \eqref{first-momentum} and \eqref{second-momentum} are satisfied.
	\end{itemize}
\end{lemma}

\begin{proof}
	Since the identities \eqref{first-momentum} and \eqref{second-momentum} hold true for smooth and rapidly decaying solutions, their extension to solutions with initial data in $L^1(\langle x \rangle^2 dx)$ follows from a standard density argument that we omit.
	
	We now prove the instantaneous boundedness and decrease of the free energy. All the manipulations below are justified by the parabolic regularization property established in Theorem~\ref{th:lwp:L1}, together with the strict positivity $u>0$, which follows from the maximum principle for parabolic equations with nonnegative initial data. By continuity of the flow in $L^1(\langle x \rangle^2 dx)$, together with the $L^\infty$ regularization ensured by Theorem~\ref{th:lwp:L1}, Lemma~\ref{lem:free-energy-well-def} implies that the free energy is well-defined for all $t \in (0,T(u_0))$.
	
	Fix $0<t'<t''<T$ and consider for $R>0$ the localized free energy functional
	\[
	\mathcal F_R[u](t) = \int \chi_R u \log u \;-\; \frac{1}{2} \int \chi_R u \, \Phi_{ u}.
	\]
	Differentiating the interaction term, using $-\Delta \Phi_f=f$, gives
	\begin{align*}
		\frac{d}{dt}\int u \Phi_{u}\chi_{R}
		&= 2 \int \partial_{t}u \, \Phi_{u}\chi_{R} 
		+ \int \Phi_u \, \Phi_{\partial_{t}u}\Delta \chi_{R} 
		+ 2 \int \nabla \Phi_u  \Phi_{\partial_{t}u} \cdot \nabla \chi_{R}.
	\end{align*}
	As a consequence, using the Keller-Segel equation $u_t=\nabla \cdot[u(\nabla \log u - \nabla \Phi_{u})]$,
	\begin{align*}
	\frac{d}{dt}\mathcal{F}_{R}[u](t)
	& = \int (1+\log u - \Phi_{u})\, \partial_{t}u \, \chi_{R}
	-\frac 12 \int u \, \Phi_{\partial_{t}u}\Delta \chi_{R} 
		- \int \nabla \Phi_u  \Phi_{\partial_{t}u} \cdot \nabla \chi_{R}\\
	&= - \int u \, |\nabla \log u - \nabla \Phi_{u}|^{2}\chi_{R}+B_R(t)
	\end{align*}
	where, using $\Phi_{\partial_{t}u}=-u-\nabla \cdot \Phi_{u\nabla \Phi_u}$,
	$$
	B_R(t)= \int  \, (u+\nabla \cdot \Phi_{u\nabla \Phi_u})( \frac 12 \Phi_u\Delta \chi_{R} + \nabla \Phi_u  \cdot \nabla \chi_{R})+\int u (\log u-\Phi_u)(\Delta \chi_R-\nabla \Phi_u\cdot \nabla \chi_R)
	$$
	gathers the boundary terms. To estimate $B_R$ on $[t',t'']$, since $u(t)\in C([t',t''],L^1(\langle x \rangle^2 dx))$ we have that $u$ is uniformly bounded in $L^1(\langle x \rangle^2 dx)$ on that time interval. Moreover, $\nabla^k u$ is bounded on $[t',t'']\times \R^2$ for every $k\in \N$. By Lemma \ref{lem:free-energy-well-def}, $u\Phi_u$ and $u\log u$ are uniformly bounded in $L^1$. Moreover, by bounding
	$$
	|\nabla \Phi_f(x)|\lesssim \int_{\langle x-y\rangle<\langle x \rangle/2}|f(y)| \frac{dy}{|x-y|}+ \int_{\langle x-y\rangle>\langle x \rangle/2}|f(y)| \frac{dy}{|x-y|}\lesssim \langle x \rangle^{-1} (\| f\|_{L^1(\langle x \rangle^2dx)}+\| f\|_{L^\infty}),
	$$
	we see $|\nabla \Phi_u|,|\nabla \cdot \Phi_{u\nabla \Phi_u}|\lesssim \langle x \rangle^{-1}$ uniformly. In addition, $|\nabla \chi_R|\lesssim R^{-k}$ and is supported in $\{R<|x|<2R\}$. Combining all these estimates imply that $B_R(t)=o_{R\to \infty}(1)$ uniformly on $[t',t'']$. We conclude the proof of the dissipation identity \eqref{freeenergy-2} by passing to the limit.
	
\end{proof}

	\section{The Asymptotic Trichotomy of the 2D parabolic-elliptic Keller-Segel System}\label{sec:classification}

In this section, we prove Theorems \ref{th:subcritical-measures}, \ref{th:supercritical} and \ref{thm:classification}. We begin with the subcritical case.
\begin{proof}[Proof of Theorem \ref{th:subcritical-measures}]
Since \(M<8\pi\), every atom of \(u_0\) has mass strictly smaller than \(8\pi\). Hence applying for instance Theorem 1 (i) in \cite{BM} for measure-valued data, and there exists a unique local mild solution with initial datum \(u_0\). Moreover, this solution regularizes instantaneously and for every \(t>0\),
    $u(\tau)\in L^1(\mathbb R^2)\cap L^\infty(\mathbb R^2)$,
     and 
    $\|u(\tau)\|_{L^1(\mathbb R^2)}=M$.
We now fix \(t>0\). The proof follows applying the classification proved in \cite{N} to the solution starting from the datum \(u(t)\in L^1(\mathbb R^2)\) whose mass is still \(M<8\pi\).
\end{proof}
Before proceeding to the proof of Theorem~\ref{th:supercritical}, it is convenient to establish the following lemma, which follows from a review of some results proved in \cite{M}.
\begin{lemma}\label{lem:outer-linfty}
Let \(u\) be a radial finite-time blow-up solution as in \cite{M}, and let
\(\lambda(t)\) be the blow-up scale. Then
\[
\lim_{A\to\infty}\limsup_{t\to T}
\sup_{|x|\ge A\lambda(t)}
\lambda^2(t)u(x,t)=0.
\]
\end{lemma}

\begin{proof}
We use the partial mass formulation described in Remark \ref{Partial Mass Formulation}. 
By Lemma 2.7 in \cite{M}, for every fixed \(R_0>0\), there exists
\(\tau_0\) such that \(w(\cdot,\tau)\) is nonincreasing on \([0,R_0]\) for
\(\tau\ge \tau_0\).

Fix \(R_0>0\). We first estimate the region
\[
A\lambda(t)\le |x|\le R_0\sqrt{T-t}.
\]
Let \(\xi=|x|\) and \(r=\xi/\sqrt{T-t}\). Since \(w_r\le0\) on
\([0,R_0]\), identity (3.7) in \cite{M} gives
\[
e^{-\tau} u(\xi,t)
=
r w_r(r,\tau)+2w(r,\tau)
\le
2w(r,\tau).
\]
Therefore
\[
u(\xi,t)
\le
\frac{2}{T-t}w(r,\tau).
\]
Using again that \(w(\cdot,\tau)\) is nonincreasing, for
\(\xi\ge A\lambda(t)\) we get
\[
w(r,\tau)
\le
w\!\left(\frac{A\lambda(t)}{\sqrt{T-t}},\tau\right).
\]
Hence
\[
\lambda^2(t)u(\xi,t)
\le
\frac{2\lambda^2(t)}{T-t}
w\!\left(\frac{A\lambda(t)}{\sqrt{T-t}},\tau\right).
\]
By the definition of \(w\) and \(W\) in Remark \ref{Partial Mass Formulation},
\[
w\!\left(\frac{A\lambda(t)}{\sqrt{T-t}},\tau\right)
=
(T-t)
\frac{M(A\lambda(t),t)}{A^2\lambda^2(t)}.
\]
Consequently,
\[
\lambda^2(t)u(\xi,t)
\le
\frac{2}{A^2}M(A\lambda(t),t).
\]
Since the total mass is conserved,
\[
M(A\lambda(t),t)
\le
\frac1{2\pi}\int_{\mathbb R^2}u(x,t)\,dx
=
\frac{M}{2\pi}.
\]
Thus
\begin{equation} \label{bd1}
\lim_{A\to \infty} \limsup_{t\to T}\sup_{A\lambda(t)\le |x|\le R_0\sqrt{T-t}}
\lambda^2(t)u(x,t)
\le \limsup_{A\to \infty}
\frac{M}{\pi A^2} =0.
\end{equation}
It remains to estimate the exterior region
\[
|x|\ge R_0\sqrt{T-t}.
\]
Taking \(R_0=2\), the argument in the proof of Proposition 3.2 in \cite{M},
based on Lemma 2.6 and identity (3.7), gives
\[
(T-t)u(x,t)\le C
\qquad
\text{for } |x|\ge 2\sqrt{T-t},
\]
for \(t\) sufficiently close to \(T\). Hence, as $\lambda^2/(T-t)\to 0$ as $t\to T$,
\begin{equation} \label{bd2}
\limsup_{t\to T}
\sup_{|x|\ge 2\sqrt{T-t}}
\lambda^2(t)u(x,t)=0.
\end{equation}
Combining the estimates \eqref{bd1} and \eqref{bd2} in the two regions shows the desired result. 
\end{proof}
\begin{proof}[Proof of Theorem \ref{th:supercritical}] Assume \(u_0\) is radial. For every
	\(t_0\in(0,T)\), Theorem \ref{th:lwp:L1} implies that
	$u(t_0)\in L^1(\mathbb R^2)\cap L^\infty(\mathbb R^2)$,
	and the solution is classical on \((t_0,T)\). Therefore the classification in \cite{M} applies to the solution restarted at time \(t_0\),
	yielding the desired decomposition
	together with the asymptotic law for \(\lambda(t)\). It remains to justify the two convergence
estimates \eqref{bd:reminder-supercritical-result} for the remainder.

Recall that, by Theorem 1.1 and Proposition 4.3 in \cite{M}, with the scale
\(\lambda(t)\) given by \eqref{id:scale-supercritical-result}, one has
\begin{align}\label{Miz:locconvergence}
\lambda^2(t) u(\lambda(t)y,t)\to U(y)
\end{align}
locally uniformly in \(y\in\mathbb R^2\). Moreover, by Lemma 2.1 (ii) in
\cite{M},
\begin{align}\label{Miz:measconvergence}
u(\cdot,t)\rightharpoonup 8\pi\delta_0+F
\qquad\text{as }t\to T,
\end{align}
with \(F\in L^1(\mathbb R^2)\), while Proposition 2.2 in \cite{M} gives the
mass quantization at the blow-up scale,
\begin{align}\label{Miz:quantizationmass}
\lim_{t\to T}
\int_{|x|<\ell\sqrt{T-t}}u(x,t)\,dx
=
8\pi
\qquad
\text{for every }\ell>0.
\end{align}
\smallskip

\noindent \underline{Proof of the first estimate in \eqref{bd:reminder-supercritical-result}.} We first prove
\begin{align}\label{thm:superL1conv}
\lim_{R\to0}\sup_{0\le t<T}
\|\tilde u(t)\|_{L^1(B_R)}=0.
\end{align}
We begin with times close to \(T\). Fix \(A>1\). For \(t\) sufficiently close to
\(T\), decompose
\[
B_R
=
B_{A\lambda(t)}
\cup
\bigl(B_R\setminus B_{A\lambda(t)}\bigr).
\]
Using \eqref{Miz:locconvergence}, we get
\[
\int_{B_{A\lambda(t)}}
\left|
u(x,t)
-
\frac1{\lambda^2(t)}
U\!\left(\frac{x}{\lambda(t)}\right)
\right|dx
=
\int_{B_A}
\left|
\lambda^2(t)u(\lambda(t)y,t)-U(y)
\right|dy
\to0
\]
as \(t\to T\), for every fixed \(A\). On the complementary region
\[
\int_{B_R\setminus B_{A\lambda(t)}}|\tilde u(x,t)|dx
\le
\int_{B_R\setminus B_{A\lambda(t)}}u(x,t)dx
+
\int_{B_R\setminus B_{A\lambda(t)}}
\frac1{\lambda^2(t)}
U\!\left(\frac{x}{\lambda(t)}\right)dx.
\]
The second term satisfies
\[
\int_{B_R\setminus B_{A\lambda(t)}}
\frac1{\lambda^2(t)}
U\!\left(\frac{x}{\lambda(t)}\right)dx
=
\int_{B_{R/\lambda(t)}\setminus B_A}U(y)dy
\le
\int_{\mathbb R^2\setminus B_A}U(y)dy.
\]
For the first term, choose \(\ell>0\). Then
\[
\int_{B_R\setminus B_{A\lambda(t)}}u
\le
\int_{B_{\ell\sqrt{T-t}}\setminus B_{A\lambda(t)}}u
+
\int_{B_R\setminus B_{\ell\sqrt{T-t}}}u.
\]
By \eqref{Miz:quantizationmass} and \eqref{Miz:locconvergence}, we have
\[
\limsup_{t\to T}
\int_{B_{\ell\sqrt{T-t}}\setminus B_{A\lambda(t)}}u
\le
8\pi-\int_{B_A}U(y)dy=\int_{\R^{2}\setminus B_{R}}U(y)dy.
\]
Moreover, by \eqref{Miz:measconvergence} and \eqref{Miz:quantizationmass},
\[
\limsup_{t\to T}
\int_{B_R\setminus B_{\ell\sqrt{T-t}}}u
\le
\int_{B_R}F(x)dx.
\]
Combining the previous estimates yields
\[
\limsup_{t\to T}\|\tilde u(t)\|_{L^1(B_R)}
\le
2\int_{\mathbb R^2\setminus B_A}U(y)dy
+
\int_{B_R}F(x)dx.
\]
Letting first \(A\to\infty\) and then \(R\to0\), and using
\(U,F\in L^1(\mathbb R^2)\), we obtain
\begin{equation} \label{bd:supercritical-mass1}
\lim_{R\to0}\limsup_{t\to T}
\|\tilde u(t)\|_{L^1(B_R)}=0.
\end{equation}
It remains to make the estimate uniform on the whole interval \([0,T)\). Fix
\(\delta>0\). On the compact interval \([0,T-\delta]\), the solution is smooth and
locally integrable uniformly in time. Therefore
\begin{equation} \label{bd:supercritical-mass2}
\lim_{R\to0}
\sup_{0\le t\le T-\delta}
\|\tilde u(t)\|_{L^1(B_R)}
=0.
\end{equation}
Combining \eqref{bd:supercritical-mass1} and \eqref{bd:supercritical-mass2} proves the desired estimate \eqref{thm:superL1conv}.

\smallskip

\noindent \underline{Proof of the second estimate in \eqref{bd:reminder-supercritical-result}.} We now prove the \(L^\infty\)-estimate
\begin{align}\label{thm:superLinfconv}
\lambda^2(t)\|\tilde u(t)\|_{L^\infty(\mathbb R^2)}\to0
\qquad\text{as }t\to T.
\end{align}
By definition,
\[
\lambda^2(t)\tilde u(x,t)
=
\lambda^2(t)u(x,t)
-
U\!\left(\frac{x}{\lambda(t)}\right).
\]
Let \(A>1\). We split
$\mathbb R^2
=
B_{A\lambda(t)}
\cup
\bigl(\mathbb R^2\setminus B_{A\lambda(t)}\bigr)$, and bound accordingly
\begin{equation} \label{id:supercritical-decomposition}
\begin{aligned}
\limsup_{t\to T}
\lambda^2(t)\|\tilde u(t)\|_{L^\infty}
&\le
\limsup_{t\to T}
\sup_{|x|\le A\lambda(t)}
\left|
\lambda^2(t)u(x,t)
-
U\!\left(\frac{x}{\lambda(t)}\right)
\right|
\\
&\quad+
\limsup_{t\to T}
\sup_{|x|\ge A\lambda(t)}
\lambda^2(t)u(x,t)
+
\sup_{|y|\ge A}U(y).
\end{aligned}
\end{equation}
In the inner region, \eqref{Miz:locconvergence} shows that for any $A>1$ the first term in \eqref{id:supercritical-decomposition} is zero,
\begin{align}\label{thm:superInnConv}
\sup_{|x|\le A\lambda(t)}
\left|
\lambda^2(t)u(x,t)
-
U\!\left(\frac{x}{\lambda(t)}\right)
\right|
=
\sup_{|y|\le A}
\left|
\lambda^2(t)u(\lambda(t)y,t)-U(y)
\right|
\to0
\end{align}
as $t\to T$. Lemma~\ref{lem:outer-linfty} shows that the second term in \eqref{id:supercritical-decomposition} converges to zero as $A\to \infty$, while the third term converges to zero too as $A\to \infty$ by the decay of $U$. Combining, we obtain the desired bound \eqref{thm:superLinfconv}.

\end{proof}
We are now left with the proof of our main result, namely Theorem~\ref{thm:classification}.
\begin{proof}[Proof of Theorem \ref{thm:classification}]

	The characterization of the maximal time of existence follows from {\cite{W}}: solutions are global if and only if \(M\le 8\pi\),
	while finite-time blow-up occurs if \(M>8\pi\). We now consider each mass regime separately.

    \smallskip
    
	\noindent \emph{Subcritical Regime.} In the subcritical regime \(M<8\pi\), the asymptotic convergence toward the unique self-similar profile follows from Theorem \ref{th:subcritical-measures}. Under the additional assumption of finite second moment, convergence rates can be made quantitative by combining the spectral stability result of \cite{EM} with interpolation arguments. More precisely, introducing the self-similar variables
\[
g(s,y)=(1+t)\,u(t,x), \qquad
y=\frac{x}{\sqrt{1+t}}, \qquad
s=\frac12\log(1+t),
\]
and denoting by \(G=\Psi_M\) the stationary self-similar profile, the arguments in \cite{EM} gives
\[
\|g(s)-G\|_{L^{4/3}}
\le C_{\eta} e^{-\eta s}
\]
for some \(\eta>0\). We split
\[
\|g(s)-G\|_{L^1}
\le
\int_{|y|\le R}|g-G|\,dy
+
\int_{|y|>R}|g-G|\,dy.
\]
The first term is estimated by Hölder's inequality,
\[
\int_{|y|\le R}|g-G|\,dy
\le
C R^{1/2}\|g-G\|_{L^{4/3}},
\]
while the second one is controlled through the uniform second moment bound,
\[
\int_{|y|>R}|g-G|\,dy
\le
\frac{C}{R^2}.
\]
Optimizing in \(R\) yields
\[
\|g(s)-G\|_{L^1}
\le
C e^{-\alpha s}
\]
for some \(\alpha>0\). Furthermore, Theorem~1.4 of \cite{EM} provides the uniform bound 
$\sup_{s\ge s_0}\|g(s)\|_{W^{2,\infty}}<\infty$,
and therefore
\[
\sup_{s\ge s_0}\|\nabla(g(s)-G)\|_{L^\infty}<\infty.
\]
Applying the Gagliardo--Nirenberg interpolation inequality
$\|h\|_{L^\infty}
\le
C
\|h\|_{L^1}^{1/3}
\|\nabla h\|_{L^\infty}^{2/3}$,
with \(h=g-G\), we obtain
\[
\|g(s)-G\|_{L^\infty}
\le
C e^{-\alpha s},
\]
possibly after decreasing \(\alpha\). Finally, returning to the original variables and observing that
\[
\tilde u(t,x)
=
\frac{1}{1+t}\bigl(g(s,y)-G(y)\bigr),
\]
we conclude that
\[
\|\tilde u(t)\|_{L^1}
+
(1+t)\|\tilde u(t)\|_{L^\infty}
\le
\frac{C}{(1+t)^\alpha},
\]
for some \(\alpha>0\).

\smallskip

	\noindent \emph{Critical Regime.} In the critical regime \(M=8\pi\), the conclusion follows directly from Theorem \ref{th:critical} established in \cite{BC}.

\smallskip

   \noindent \emph{Supercritical Regime.} In the supercritical regime \(M>8\pi\), the conclusion follows directly from Theorem \ref{th:supercritical}.

\end{proof}

\end{document}